\documentclass[reqno]{article}

\usepackage{enumerate}
\usepackage{slashbox}
\usepackage{ifpdf}
\usepackage{amsmath}
\usepackage{amsfonts}
\usepackage{amssymb}
\usepackage{amsthm}
\usepackage{url}
\usepackage{setspace}
\newcommand{\ignore}[1]{}

\newcommand{\abs}[1]{\left\lvert {#1} \right\rvert}
\newcommand{\norm}[1]{\left\lVert {#1} \right\rVert}

\newcommand{\C}{{\mathbb{C}}}
\newcommand{\R}{{\mathbb{R}}}

\newcommand{\bB}{{\mathbb{B}}}

\newcommand{\sF}{{\mathcal{F}}}

\newcommand{\sH}{{\mathcal{H}}}

\newcommand{\sI}{{\mathcal{I}}}

\newtheorem{thm}{Theorem}[section]

\newtheorem{prop}[thm]{Proposition}

\newtheorem{cor}[thm]{Corollary}

\newtheorem{lemma}[thm]{Lemma}

\theoremstyle{definition}

\theoremstyle{remark}
\newtheorem{remark}[thm]{Remark}

\author{Ji\v{r}\'i Lebl\footnote{%
Department of Mathematics, University of Illinois
at Urbana-Champaign, 
Urbana, IL 61801, USA,
{\tt jlebl@math.uiuc.edu}}
\and
Daniel Lichtblau\footnote{%
Wolfram Research, Inc., 100 Trade Center Drive,
Champaign, IL 61820, USA,
{\tt danl@wolfram.com}}
}

\date{March 30, 2010}

\ifpdf
\fi

\title{Uniqueness of certain polynomials constant on a line}

\begin{document}

%\doublespace

\maketitle

\begin{abstract}
We study a question with connections to linear algebra, real algebraic geometry,
combinatorics, and complex analysis.
Let $p(x,y)$ be a polynomial of degree $d$ with $N$ positive coefficients
and no negative coefficients, such that $p=1$ when $x+y=1$.
A sharp estimate $d \leq 2N-3$ is known.  In this paper we study the $p$
for which equality holds.
We prove some new results about the form of these ``sharp'' polynomials.
Using these new results and using two independent computational
methods we
give a complete classification of these polynomials up to $d=17$.
The question is motivated by the problem of
classification of CR
maps between spheres in different dimensions.
\end{abstract}

%\enlargethispage{\baselineskip}

%%%%%%%%%%%%%%%%%%%%%%%%%%%%%%%%%%%%%%%%%%%%%%%%%%%%%%%%%%%%%%%%%%%%%%%%%%%%%

\section{Introduction} \label{section:intro}

In this paper we answer by
computational methods certain difficult questions about the set of
bivariate polynomials with nonnegative coefficients constant on a line.  
Following the notation of~\cite{DL, DLP}, let
$\sH(2,d)$ denote the set of polynomials $p(x,y)$ of degree $d$ with
nonnegative coefficients such that $p(x,y) = 1$ whenever $x+y=1$.

The condition that the coefficients are nonegative is 
motivated by a question in CR geometry, which we describe in
\S~\ref{section:CR}.  Without this condition,
the affine space $\sI$ of all polynomials of degree $d$ or less such that
$p(x,y) = 1$ whenever $x+y=1$ is easy to describe.  It is
in one to one correspondence with the vector space of polynomials of degree $d-1$
or less.
I.e.\ if $q$ is of degree $d-1$ or less
then we let $p(x,y)=q(x,y)(x+y-1)+1 \in \sI$.
When the coefficients of the polynomials are considered as variables,
$\sI$ is the solution set of a certain nonhomogeneous linear system.

On the other hand, $\sH(2,d)$ is a convex subset of $\sI$ (with nonempty interior).
It is the intersection of the 
positive cone, $\sI$, and the open set of polynomials of degree exactly $d$.
The questions we address are
difficult because we need to consider the geometry of the boundary of
$\sH(2,d)$.
When we try to answer these questions computationally,
we note that the complexity grows very fast.
We will describe two methods
that allow one to effectively answer questions 
about those polynomials satisfying a certain extremal property.
We prove certain theoretical statements of independent interest about these
extremal polynomials, which allow us to reduce the complexity of computation.

The condition that a polynomial is constant on a line leads to a system
of linear equations, and hence to a problem in linear algebra.  As we 
require the coefficients to be positive, we also get a natural formulation of
our classification question below as a linear programming problem.
Imposing an extremal
condition we will describe makes it a mixed-integer programming problem.

An important question from the point of view of CR geometry concerns 
the number of distinct monomials, $N=N(p)$,
for $p \in \sH(2,d)$.  $N$ is the minimal embedding dimension of the
associated CR map of spheres, see \S~\ref{section:CR}.
The following bound was proved in~\cite{DKR}:
\begin{equation} \label{eqmainbound}
d \leq 2N-3 .
\end{equation}
Furthermore, for all odd $d$ the functions
\begin{equation} \label{eqinvmaps}
f_d (x,y) :=
\left(\frac{x + \sqrt{x^2 + 4y}}{2}\right)^d
+ \left(\frac{x - \sqrt{x^2 + 4y}}{2}\right)^d +
(-1)^{d+1} y^d 
\end{equation}
are in fact polynomials in $\sH(2,d)$ with
$d = 2N(f_d)-3$, see \cite{D:poly}.  Thus the inequality \eqref{eqmainbound}
is sharp.

The polynomials $f_d$ have many other interesting
properties.  In particular, they are group invariant.
The CR maps that arise from $f_d$ are one of the
only two possible classes of group invariant maps of balls.
See \S~\ref{section:CR} for information and the references within.
The polynomials $f_d$ are also related to Chebychev polynomials, arise in
denesting radicals~\cite{Osler}, and have connections to number
theory~\cite{D:numbertheory, DilcherStolarsky},
combinatorics~\cite{LWW} and other fields.
For example, $f_d(x,y) = x^d+y^d \pmod{d}$
if and only if $d$ is an odd prime~\cite{D:numbertheory}.

If $p \in \sH(2,d)$ minimizes $N(p)$ for a fixed degree $d$,
then we call $p$ a \emph{sharp} polynomial.
It can be proved that for a fixed degree only finitely many sharp polynomials
exist.
We wish to ask: Is $f_d$ the unique sharp polynomial in $\sH(2,d)$
up to swapping of variables?  If not, what are
all the sharp polynomials for a given degree?

The programs used for the computations were
\textit{Mathematica} 6 \cite{Mathematica}, and
Genius 1.0.2 \cite{Genius}.  Some code was also written in
native C for speed using the GMP~\cite{GMP} library.
We used two different approaches to the problem 
to get independent verification of the result
and minimize the effects of possible bugs in the underlying code.
The computer code used is available at the url:
\url{http://www.jirka.org/LL08-archive.zip}.

Let us mention some relevant previous work.
With the aid of a computer,
Wono~\cite{Wono} classified all polynomials in $\sH(2,d)$ with 5 terms or
fewer, and hence all sharp polynomials up to
degree 7. 
The second author~\cite{Lichtblau:thesis} found new sharp polynomials
in degree 11.
D'Angelo and the first author~\cite{DL} have shown that
there are infinitely many $d$ for which there exist other
sharp polynomials.

The organization of the paper follows.
In \S~\ref{section:results} we state our main results including the list 
of all sharp polynomials in odd degree up to $d=17$.  In
\S~\ref{section:graphthry} we prove new results about the form of sharp
polynomials, which are useful in reducing the computation time.  These
results are of also
independent interest.  
In
\S~\ref{section:linprob} we describe treating coefficients
of polynomials constant on $x+y=1$ as a linear problem and prove several
related results.  In \S~\ref{section:linalgmeth} we describe the method
of finding
sharp polynomials by computing the nullspace of certain matrices.
In \S~\ref{section:linprog} we describe the mixed-integer programming method
we used to find sharp polynomials.
In \S~\ref{section:sharpmaps} and \S~\ref{section:construction}
we describe degrees in which uniqueness definitely fails by describing
a construction of new sharp polynomials.
This problem has a long and complicated history and motivation.
In \S~\ref{section:CR} we discuss
the motivation from CR geometry and complex analysis.

We thank John D'Angelo for suggesting this research 
and for helpful discussions along the way.  We would also like to thank the
referees for this and a previous draft for many useful comments and
suggestions.
The first author would like
to acknowledge MSRI and AIM for holding workshops that focused on these
and related questions.

%%%%%%%%%%%%%%%%%%%%%%%%%%%%%%%%%%%%%%%%%%%%%%%%%%%%%%%%%%%%%%%%%%%%%%%%%%%%%

\section{Main results} \label{section:results}

In this section we state and discuss our main results.
When $p \in \sH(2,d)$ is such that the number of terms $N(p)$
is minimal in $\sH(2,d)$, we say that $p$ is \emph{sharp}.
We say that \emph{uniqueness holds} for degree $d$
if there is exactly one sharp polynomial in $\sH(2,d)$ up to interchanging
the variables $x$ and $y$.  Otherwise, we say that uniqueness fails
for $d$.
Our main result is the following theorem.

\begin{thm} \label{mainthm}
Up to swapping of variables,
Table~\ref{tableofallsharp} lists all sharp polynomials of odd
degree for $d \leq 17$.  In particular, uniqueness holds for
$d=1, 3, 5, 9, 17$.
\end{thm}

In the table, the first result listed is the group invariant one.
To make the computations feasible, we prove new results about the form
of sharp polynomials in \S~\ref{section:graphthry}.  In particular,
in Lemma~\ref{topdegxdyd} we prove
that the degree $d$ terms must be $x^d+y^d$,
and these must be the only pure terms.
We also prove that certain degree $d-1$ monomials cannot arise, and that at
least one degree $d-1$ monomial is present.

For degree $d=19$ it is currently computationally infeasible to run the
tests.  Uniqueness does not hold in degree 19 by explicit construction
(see \S~\ref{section:sharpmaps} and \S~\ref{section:construction})
as $19 \equiv 3 \pmod{4}$ and
$19 \equiv 1 \pmod{6}$.  Hence, there are at least two inequivalent sharp
polynomials in $\sH(2,d)$ apart from the group invariant $f_d$.
When $d=21$ it is unknown if uniqueness holds, although the method
described in \S~\ref{section:construction} does not produce any new sharp
polynomials for this degree.

To construct sharp polynomials of even degrees, we take two
sharp polynomials of odd degree $p$ and $q$, and write $p(x,y) = x^d + y^d +
p_0(x,y)$ utilizing Lemma~\ref{topdegxdyd}.  We let
\begin{equation} \label{allevenex}
f(x,y) := x^d + y^d q(x,y) + p_0(x,y) ,
\end{equation}
and we note that $f(x,y)$ is sharp and $\deg f = \deg p + \deg q$.

We have a computer assisted proof of the following theorem.
Note that in the
even degree case we can no longer make simplifying assumptions about terms of
degree $d$ or $d-1$.
Therefore, the computation is more expensive.
The only assumption
we can make is to use Lemma~\ref{weakxkym} to force exactly 1 pure term
in each variable.

\begin{thm} \label{thm:evendegree}
For $2 < d \leq 12$, $d$ even, the procedure in \eqref{allevenex}
generates all sharp polynomials of degree $d$ up to swapping of variables.
\end{thm}

In degree $2$,
the homogeneous polynomial $(x+y)^2 = x^2 + 2xy + y^2$ is sharp.
This polynomial is 
the only known case of a sharp polynomial with 3 terms of top degree.
The only other polynomial (up to swapping of variables) in degree $2$
is the so-called Whitney map $x+y(x+y) = x+xy+y^2$.

We summarize results for the number of polynomials in $\sH(2,d)$
where $d=2N-3$,
$d=2N-4$, and $d=2N-5$ in Table~\ref{tablenums}.
Here we do not identify polynomials up to swapping of variables.
The first row is the number of sharp polynomials for odd degrees,
the second row is the number of sharp polynomials in even degrees.
We note that verifying the first two rows of
the table essentially amounts to proving
Theorem~\ref{thm:evendegree}.
For degree $d=2N-6$, with $N \geq 4$, there always exists a one parameter
family of polynomials in $\sH(2,d)$ and hence the number is always infinite.
It is unknown if there are any one parameter families when degree is $d=2N-5$,
that is, if the number of such polynomials is infinite for some $N$.
A negative answer to
this question is
equivalent to proving Proposition~\ref{familyest} for all degrees.

We pose several open questions
about the sequences corresponding
to rows of Table~\ref{tablenums}.  Denote these sequences
by $a_N$, $b_N$, and $c_N$ corresponding to the first, second, and third row
respectively.
Is $a_N$ bounded?  Both $b_N$ and $c_N$ are unbounded, but we know little about their rates of growth.
Is $\frac{a_N}{N}$, $\frac{b_N}{N}$ or $\frac{c_N}{N}$ bounded?
Is $b_N$ or $c_N$ monotone?
If Theorem~\ref{thm:evendegree} is true for all degrees, then $b_N$
must be monotone, as we can derive a formula for $b_N$ in terms of the first
row of the table (with the exception of $b_3$).  That is,
$b_N = 2(a_2a_{N-1}+a_3a_{N-1}+\cdots + a_{N-1}a_2)$.  None
of the three sequences previously appeared in The
On-Line Encyclopedia of Integer Sequences~\cite{Sloane}, and were
entered as A143107, A143108 and A143109.

The results above, together with the results given in
\S~\ref{section:sharpmaps} and \S~\ref{section:construction} also suggest the
following conjectures.

\begin{enumerate}[(i)]
\item
Up to swapping of variables, $f_d$ is the
unique sharp polynomial in $\sH(2,d)$ for infinitely many odd degrees $d$.
\item
All sharp polynomials in even degrees greater than two are constructed from 
sharp polynomials in smaller odd degrees by the procedure in
\eqref{allevenex}.
\end{enumerate}

\begin{table}[h!]
\caption{Complete list of sharp polynomials of odd degree up to $17$.
Degrees in which uniqueness holds are marked with an asterisk.}
 \label{tableofallsharp}
\begin{center}
\begin{tabular}{l|l}
$d$ & Sharp polynomials \\[1pt]
\hline
\\[-8pt]
$1^*$  & $x+y$\\[5pt]
$3^*$  & $x^3+3xy+y^3$\\[5pt]
$5^*$  & $x^5+5x^3y+5xy^2+y^5$\\[5pt]
7  & $x^7 + 7x^3y + 14x^2y^3 + 7xy^5 + y^7$\\[2pt]
    & $x^7 + 7x^3y + 7x^3y^3 + 7xy^3 + y^7$\\[2pt]
    & $x^7 + \frac{7}{2}x^5y + \frac{7}{2}xy + \frac{7}{2}xy^5 + y^7$\\[5pt]
$9^*$  & $x^9 + 9x^7y + 27x^5y^2 + 30x^3y^3 + 9xy^4 + y^9$\\[5pt]
11  & $x^{11} + 11x^9y + 44x^7y^2 + 77x^5y^3 + 55x^3y^4 + 11xy^5 + y^{11}$\\[2pt]
    & $x^{11} + 11x^5y + 11x^5y^5 + 55x^4y^3 + 55x^3y^5 + 11xy^5 + y^{11}$\\[5pt]
13  & $x^{13} + 13 x^{11}y + 65x^9y^2 + 156x^7y^3 + 182x^5y^4
       + 91x^3y^5 + 13xy^6 + y^{13}$\\[2pt]
    & $x^{13} + 13x^{11}y + 65x^9y^2 + \frac{221}{2}x^7y^3 + \frac{92}{2}x^3y^3
       + \frac{91}{2}x^3y^7 + 13xy^6 + y^{13}$\\[2pt]
    & $x^{13} + \frac{234}{25}x^{11}y + \frac{143}{5}x^8y^2
       + \frac{143}{5}x^7y^4 + \frac{91}{25}xy + \frac{143}{25}xy^6
       + \frac{91}{25}xy^{11} + y^{13}$\\[2pt]
    & $x^{13} + \frac{234}{25}x^{11}y + \frac{143}{5}x^9y^2
       + \frac{143}{5}x^7y^3 + \frac{91}{25}xy + \frac{143}{25}xy^6
       + \frac{91}{25}xy^{11} + y^{13}$\\[5pt]
15  & $x^{15} + 15x^{13}y + 90x^{11}y^2 + 275x^9y^3
       + 450x^7y^4 + 378x^5y^5 + 140x^3y^6$\\[2pt]
    &  $~~+ 15xy^7 + y^{15}$\\[2pt]
    & $x^{15} + 140x^9y^3 + 15x^7y + 420x^7y^4
       + 15x^7y^7 + 378x^5y^5 + 140x^3y^6$\\[2pt]
    & $~~+ 15xy^7 + y^{15}$\\[5pt]
$17^*$ & $ x^{17} + 17 x^{15}y + 119 x^{13}y^2 + 442 x^{11}y^3 + 935 x^9y^4
      + 1122 x^7y^5$\\[2pt]
    & $~~+ 714 x^5y^6 + 204 x^3y^7 + 17 xy^8 + y^{17}$
\end{tabular}
\end{center}
\end{table}

\begin{table}[h!]
\caption{Number of polynomials in the top 3 degrees for each $N$.}
 \label{tablenums}
\begin{center}
\begin{tabular}{l|r|r|r|r|r|r|r|r|r}
\backslashbox{degree}{$N$} & $2$ & $3$ & $4$ & $5$ & $6$ & $7$ & $8$ & $9$ & $10$
\\
\hline
 & & & & & & & & & \\[-8pt]
$d=2N-3$ & 1 & 1 & 2 & 4 & 2 & 4 & 8 & 4 & 2 \\
$d=2N-4$ & 0 & 3 & 4 & 10 & 24 & 32 & 56 & ? & ? \\
$d=2N-5$ & 0 & 0 & 11 & 38 & 88 & 198 & ? & ? & ? \\
\end{tabular}
\end{center}
\end{table}

%%%%%%%%%%%%%%%%%%%%%%%%%%%%%%%%%%%%%%%%%%%%%%%%%%%%%%%%%%%%%%%%%%%%%%%%%%%%%

\section{Form of sharp polynomials} \label{section:graphthry}

In this section we prove new results about the form of sharp polynomials.
These results are proved by extending the graph theoretic proof
of the inequality $d \leq 2N-3$ from~\cite{DKR}.  We 
therefore sketch the main ideas of that proof as we will need them.
Our primary application of these results is to reduce the computation time
of our computer code, but they are also of independent interest.
The main result of this section is the following lemma.

\begin{lemma} \label{topdegxdyd}
Let $d$ be an odd integer and let $f \in \sH(2,d)$ be sharp.
Then,
\begin{equation} \label{eq:topdegxdyd}
f(x,y) = x^d + y^d + \text{(lower order terms)}.
\end{equation}
\end{lemma}

We remark that in general it is not too hard to prove that any $f \in \sH(2,d)$
must have at least two terms of degree $d$,
and such a statement can be generalized
to higher dimensions as well.  Also the reader should notice that the lemma
cannot possibly hold for $d$ even, see~\eqref{allevenex}.
In fact, there exists no polynomial in $\sH(2,d)$ of even degree
of the form~\eqref{eq:topdegxdyd}.  Writing $f(x,y) -1 = (x+y-1) q(x,y)$,
we notice that the top degree terms of $f$ must be divisible by $(x+y)$
and $x^d+y^d$ is only divisible by $x+y$ if $d$ is odd.

\begin{prop} \label{onlypureterms}
Suppose $f \in \sH(2,d)$ and $f(x,y) = x^d + y^d + g(x,y)$,
where $g$ is of degree less than
$d$.  Then $g(0,y) \equiv g(x,0) \equiv 0$.
\end{prop}

In other words, if $f$ is of the form \eqref{eq:topdegxdyd},
then the lower order terms involve only mixed terms.
By mixed terms we mean terms involving both $x$ and $y$.

\begin{proof}
Notice that $f(1,0) = 1$.  Then it is obvious that $g(1,0) = 0$.
As the coefficients of $g$ are positive, we get $g(x,0) \equiv 0$.
\end{proof}

A further corollary of Lemma~\ref{topdegxdyd} also reduces
the search space:

\begin{lemma} \label{dminus1}
Let $d > 1$ be an odd integer and let $f \in \sH(2,d)$ be sharp.  Then at least
one monomial of degree $d-1$ has a nonzero coefficient.
\end{lemma}

\begin{proof}
Write $f(x,y) -1 = (x+y-1) q(x,y)$ for some $q$.  Write $q = q_{d-1}+q_{d-2}+\cdots+q_0$ as the homogeneous decomposition of $q$.  If $f$ does not
have any monomials of degree $d-1$ then $(x+y)q_{d-2} - q_{d-1} = 0$.
Hence, the top degree terms of $f$
must be divisible by $(x+y)^2$.  By Lemma~\ref{topdegxdyd} the top degree
part of a sharp polynomial is $x^d+y^d$,
which is not divisible by $(x+y)^2$.
\end{proof}

Further analysis of the proof of~\cite{DKR} reveals the
following lemma, which further reduces the search space.

\begin{lemma} \label{dminus1unreachable}
Let $d > 1$ be an odd integer and let $f \in \sH(2,d)$ be sharp.
Then $f$ does not contain terms of the form
$x^jy^{d-1-j}$ for even $j$.
\end{lemma}

Before we prove Lemma~\ref{topdegxdyd} and Lemma~\ref{dminus1unreachable},
we need to set up the terminology
of~\cite{DKR} and restate some of their results.
First we have the following
characterization of homogeneous polynomials in $\sH(2,d)$.  More
general related results were proved in \cite{Rudin84} and \cite{D:book}.
For convenience of 
the reader we prove the following special case,
in the setting of that 
applies to this work.

\begin{prop} \label{homogcase}
Let $h \in \sH(2,d)$ be homogeneous.  Then $h(x,y) = (x+y)^d$.
\end{prop}

\begin{proof}
Note that $(x+y)^d = h(x,y)$ when $x+y=1$.  Any point in the first quadrant
of $\R^2$ can be written as $(tx,yx)$ where $x+y=1$ an $t \geq 0$.
By homogeneity, $(tx+ty)^d = t^d(x+y)^d= t^d h(x,y) = h(tx,ty)$.
Two polynomials equal on an open set are equal everywhere.
\end{proof}

Multiplying any lower degree part of a polynomial in $\sH(2,d)$ by $(x+y)^k$
does
not get us out of the space.  Hence, for $f \in \sH(2,d)$ we write
the homogeneous decomposition $f = f_d + f_{d-1} + \cdots + f_0$.  Using
Proposition~\ref{homogcase} we get
\begin{equation} \label{eq:undoing}
(x+y)^d = f_d(x,y) + (x+y)f_{d-1} + (x+y)^2f_{d-2} + \cdots + (x+y)^d f_0 .
\end{equation}
Therefore, every polynomial in $\sH(2,d)$ is constructed by starting with
$(x+y)^d$, partitioning it into two parts and dividing one by $(x+y)$, then repeating the process.  This operation is called \emph{undoing}.

Write
\begin{equation} \label{diveq}
f(x,y) - 1 = (x+y-1) q(x,y) .
\end{equation}
We study the coefficients of $q$.  In particular we write the Newton diagram
for $q$ where we ignore the size of each coefficient and only write $P$, $N$
or $0$ for positive, negative or zero respectively.
For example, when $f = x^3+3xy+y^3$,
then $q = {y}^{2}+y-xy+{x}^{2}+x+1$ and the diagram
(including the corresponding monomials) is
\begin{equation} \label{eq:examplenewton}
\begin{bmatrix}
x^3 & \mathbf{0} & 0 & 0 & 0\\
x^2 & P & 0 & 0 & 0 \\
x   & P & \mathbf{N} & 0 & 0 \\
1   & P & P & P & \mathbf{0}\\
    & 1 & y & y^2 & y^3
\end{bmatrix} .
\end{equation}
We have highlighted the entries corresponding to terms of $f$.  
For every entry in the diagram we define
the 2 by 2 submatrix that includes the entry itself, the entry just
below, and the entry to the left.  If the submatrix is any of the
following, then we say the entry is a \emph{sink}.
\begin{equation}
\begin{bmatrix}
P & N \\
* & P
\end{bmatrix}
,
\begin{bmatrix}
0 & N \\
* & P
\end{bmatrix}
,
\begin{bmatrix}
P & N \\
* & 0
\end{bmatrix}
,
\begin{bmatrix}
0 & N \\
* & 0
\end{bmatrix}
,
\begin{bmatrix}
P & 0 \\
* & P
\end{bmatrix}
,
\begin{bmatrix}
0 & 0 \\
* & P
\end{bmatrix}
, \text{ or }
\begin{bmatrix}
P & 0 \\
* & 0
\end{bmatrix} .
\end{equation}
Each sink in the diagram for $q$ must correspond to a nonzero positive
term in $f$.  There may be more positive terms in $f$ than there are sinks, but
not the other way around.  In \eqref{eq:examplenewton}, the sinks are marked in
bold.
One can make the corresponding definition of a \emph{source}, which would
force a negative term.  There must therefore be at most one
(and in fact exactly one) source
corresponding to the $-1$ in $f(x,y)-1$.

From equation \eqref{eq:undoing} we see that the corresponding
diagram for any $f \in \sH(2,d)$ is obtained by starting with the diagram
for $(x+y)^d$ and successively changing $P$'s to $N$'s or $0$'s.  The diagram
for $(x+y)^3$ is
\begin{equation}
\begin{bmatrix}
\mathbf{0} & 0 & 0 & 0\\
P & \mathbf{0} & 0 & 0 \\
P & P & \mathbf{0} & 0 \\
P & P & P & \mathbf{0}
\end{bmatrix} .
\end{equation}
We remark that diagrams we obtain by changing $P$'s to $N$'s or $0$'s
need not correspond to polynomials in $\sH(2,d)$.

D'Angelo, Kos and Riehl prove that the minimum number of sinks one can
obtain by this procedure is exactly $\lceil \frac{d+3}{2} \rceil$.  It turns
out that there exist polynomials in $\sH(2,d)$
with precisely this many nonzero terms.  Thus
for a sharp polynomial in $\sH(2,d)$ the nonzero terms correspond
exactly to sinks in its Newton diagram.
We can thus easily prove the following version of Lemma~\ref{topdegxdyd}.
This proposition was essentially proved in~\cite{DKR} but not stated
explicitly.

\begin{prop}[D'Angelo-Kos-Riehl] \label{weakxkym}
Let $f \in \sH(2,d)$ be sharp.  Then for some $k,m \leq d$,
\begin{equation}
f(x,y) = x^k + y^m + \text{(mixed terms)} .
\end{equation}
\end{prop}

\begin{proof}
Note that there must be exactly one source, hence the lower right entry 
in the diagram for $f$ must be a $P$.  In the bottom row, we note
that to only have sinks except for the one source, we must have
a row of some number of $P$'s and then all zeros.  That means that there is
exactly one sink on the bottom row.  Similarly there is exactly one sink
in the leftmost column.  Since $f$ is sharp, these sinks correspond
exactly to nonzero terms in $f$.  Thus there is at most one pure
term in $x$ and at most one pure term in $y$.  By plugging in $x=1$, $y=0$, and
vice versa, we get that the coefficient of both terms must be 1.
\end{proof}

In the procedure of undoing, we start with the diagram for $(x+y)^d$, which has
$d+1$ sinks along the main diagonal and we change $P$'s to $N$'s or $0$'s.
Of course we can change only the entries corresponding to terms of degree $d-1$
or less as the diagram corresponds to the $q$ in \eqref{diveq}, which is
of degree $d-1$.
We call a diagram $D'$ an \emph{ancestor} of $D$ if $D$ has less than
or equal number of nonzero entries ($P$'s or $N$'s) than has $D'$.

We note what can happen in this procedure to the sinks of the diagram
corresponding to $(x+y)^d$, which we call $D_h$.
The sinks may move leftward in rows, or downward in columns, in
which case the number of sinks is unchanged.  Sinks can also be created,
two sinks can coalesce into one or no sinks, or a sink can disappear.  The main
idea of the proof is essentially the following result which we state as a
lemma, and which is proved in~\cite{DKR}.  By a \emph{procedure} of getting
to a diagram $D$ of $f$ we mean a finite sequence of diagrams
$D_h = D_0 \to D_1 \to \cdots \to D_m = D$ such that $D_j$ is an ancestor of
$D_k$ whenever $j < k$.  All diagrams are ancestors of $D$ and $D_h$ is the
ancestor of all the diagrams in the sequence.

\begin{lemma}[D'Angelo-Kos-Riehl]
Let $D$ be a diagram corresponding to a sharp $f \in \sH(2,d)$.
Then there exists a procedure (as described above) such that
each diagram in the procedure has a unique source at the origin.
Furthermore, all coalescence happens on the diagonal corresponding to
terms of degree $d-1$.
In particular, there exists an ancestor diagram $D_1$ with the
same number of sinks as $D$, and with sinks corresponding only to terms of
degree $d-1$ and $d$.
\end{lemma}

\begin{proof}
For completeness we sketch one possible proof of this lemma, which is
slightly different from~\cite{DKR}.  The tedious details
require checking a finite number of cases and are left to the reader.
We work
in reverse, that is, we start with $D$ and work towards $D_1$.  Take $D$
and notice that we can change $0$'s to $P$'s or $N$'s without increasing the
number of sinks.  This can be done in such a way that
sinks in move up and to the right or disappear.  It is not hard to check now
that by setting certain $N$'s to $P$'s we can have sinks move up and to the
right until they reach the diagonal corresponding to degree $d-1$ or
disappear.  This
diagram is the $D_1$ we seek.
\end{proof}

We are now ready to prove Lemma~\ref{topdegxdyd}.  Intuitively the idea of
the proof is that if all the coalescence happens on the diagonal
corresponding to terms of degree $d-1$ then
the coalescence of sinks must happen in
matched pairs.  An extra sink is left over on each side of a row of such
matched pairs.  To minimize the number of sinks, there must be at most two
such extra left over sinks and they must be the ones corresponding to $x^d$
and $y^d$.

\begin{proof}[Proof of Lemma~\ref{topdegxdyd}]
Suppose $f\in \sH(2,d)$ is sharp and $d$ is odd.
We now know that during the procedure of getting
the diagram for $f$, we have passed through a
diagram $D_1$ with sinks only on the diagonals
corresponding to degrees $d-1$ and $d$, and having the exact same number of
sinks as the number of nonzero terms in $f$.

Let us start with $D_h$.
It is easy to see that except for the first and last entry on the degree $d-1$
diagonal, maximum coalescence happens when we change every other $P$ on the
diagonal to $N$.  Changing to $0$ does not remove the sinks of degree $d$ and
hence does not create coalescence.  Changing the first or the last entry to $N$
does not create any coalescence.  Since $d$ is odd, we note that the degree
$d-1$ diagonal of the diagram consists of $PNPNP\ldots PNP$.  Any other
arrangement has too many sinks.

The $D_1$ diagram has sinks for $x^d$ and $y^d$ and no other sinks of degree
$d$.  We can no longer create or lose any sinks on the diagonal of degree
$d$.  We know
that as $f$ is sharp, sinks correspond exactly to terms in $f$.  We know
that there are at least two terms of degree $d$.
Applying Proposition~\ref{weakxkym} gives the result.
\end{proof}

\begin{proof}[Proof of Lemma~\ref{dminus1unreachable}]
We analyze the proof further.  We again find the diagram $D_1$ that
has $PNPNP \ldots PNP$ as the degree $d-1$ diagonal.
If we create a sink by changing one of the $P$'s to a zero or an $N$
we note that we have created a source, which is not allowed.
\end{proof}

%%%%%%%%%%%%%%%%%%%%%%%%%%%%%%%%%%%%%%%%%%%%%%%%%%%%%%%%%%%%%%%%%%%%%%%%%%%%%

\section{Finding sharp polynomials as a linear problem} \label{section:linprob}

We fix $d$ and we treat the coefficients of polynomials as variables.
Hence, we can treat the space of polynomials of degree $d$ or less as $\R^K$
for some large $K$.
Suppose
\begin{equation}
p(x,y) = \sum_{j+k \leq d} c_{j,k} x^j y^k .
\end{equation}
The condition that $p(x,y) = 1$ on $x+y$ is equivalent to 
\begin{equation}
p(x,1-x) = 1 .
\end{equation}
That means all the non-constant coefficients of $h(x) = p(x,1-x)$
must be zero and the constant coefficient must be 1.
We get a linear system of $d$ equations in $K$ variables and one
affine equation.  If we instead let $p(x,1-x) = c$ and let $c$ be a variable,
then we get a linear system of $d+1$ equations in $K+1$ variables.  Whenever
we find a solution $p$ to $p(x,1-x) = c$, $c \not= 0$, we get a solution to
$p(x,1-x) = 1$ by rescaling.

Now that we know that $p \in \sH(2,d)$ are solutions of a linear system
of equations, we prove some useful propositions.
We will need the following proposition from~\cite{Lebl:thesis}, which we
reprove here as the idea of the proof is important in the next section.
A generalized version of this proposition was given in~\cite{DL:families}.

\begin{prop} \label{nofamily}
If there exists a continuous map $t \mapsto p_t \in \sH(2,d)$
(a one parameter family)
and
further that $N(p_t)$ is constant for $t$ in some open interval $I$,
then $p_t$ is not sharp for $t \in I$.
\end{prop}

\begin{proof}
$p \in \sH(2,d)$ are solutions of a linear system, hence if there is
a family, then there is a straight line with the same property.
We pick two polynomials $\varphi$ and $\psi$ such that
$p_t := t\varphi+(1-t)\psi \in \sH(2,d)$ and the number of terms in $p_t$
is constant some small interval $I$.  We restrict our attention
to a closed interval $J$ where $p_t$ has nonnegative 
coefficients.  It is not hard to see that $J$ must be bounded and we could
by rescaling assume that $J=[0,1]$.  Further, we note that the number of
nonzero coefficients of $p_0$ and $p_1$ must be smaller than 
the number of coefficients in $p_t$ for $t \in (0,1)$.
Obviously $I \subset (0,1)$.  It remains to show that $p_0$ or $p_1$
are in $\sH(2,d)$.  They could conceivably be of lower degree than $d$,
but they cannot
both be such since $p_t$ is a convex combination of them.
\end{proof}

We define a \emph{support} as a subset of the set of
monomials of degree at most $d$.  The support of a polynomial
$p$ is the set of monomials with nonzero coefficients.

\begin{prop} \label{onesig}
No two sharp polynomials in $\sH(2,d)$ have the same support.
In particular, there can be at most finitely many sharp polynomials
in $\sH(2,d)$.
\end{prop}

\begin{proof}
If $p$ and $q$ have the same support, take the combination
$tp+(1-t)q$, which is a one parameter family of same support.
By Proposition~\ref{nofamily} either $p=q$ or neither $p$ nor $q$ can be
sharp.
\end{proof}

\begin{cor}
The coefficients of sharp polynomials are rational.
\end{cor}

We omit the proof as we never use this result.

%%%%%%%%%%%%%%%%%%%%%%%%%%%%%%%%%%%%%%%%%%%%%%%%%%%%%%%%%%%%%%%%%%%%%%%%%%%%%

\section{Linear algebra method} \label{section:linalgmeth}

For $j,k$ satisfying 
$0 \leq j+k \leq d-1$,
we define 
polynomials 
\begin{equation} \label{eqbasis}
b_{j,k}(x,y) := x^jy^k - x^jy^k(x+y)^{d-j-k} .
\end{equation}
Each equals $0$ on the line $x+y = 1$, and they are linearly
independent.  By counting
dimensions it is not hard to see that the polynomials
\eqref{eqbasis} together with $(x+y)^d$ span the space
of all polynomials that are constant on the line $x+y = 1$.
\eqref{eqbasis} helps us construct the actual system of linear equations
to find sharp polynomials in $\sH(2,d)$ in a relatively simple way.

We decompose $p \in \sH(2,d)$
using the basis \eqref{eqbasis} together with $(x+y)^d$.
\begin{equation}
p(x,y) = (x+y)^d + \sum_{j,k} c_{j,k} b_{j,k} (x,y) .
\end{equation}
We note that $c_{j,k}$ are precisely the coefficients of $p$ and we also
note that since $p(x,y) = 1$ on $(x+y) = 1$, then the coefficient of
$(x+y)^d$ in the decomposition must be 1.  Hence the coefficients of
$p$ degree $d$ are affine functions of the coefficients of degree $d-1$ or
less.  If we treat the coefficient of $(x+y)^d$ (and thus the
value of $p$ on the line $x+y=1$) as a variable we find that
the coefficients of $p$ are linear functions.

We will always assume that $c_{0,0} = 0$ even if looking for
nonsharp polynomials in $\sH(2,d)$.
If $p(0,0) \not= 0$ then $\frac{p(x,y) - p(0,0)}{1-p(0,0)} \in \sH(2,d)$.
If we are only looking for sharp polynomials,
then $c_{0,0}$ will always be zero.

We construct a matrix $A$ with each column corresponding to
one $b_{j,k}$ and one column corresponding to $(x+y)^d$.  Each row 
represents one monomial of degree $d$.  This matrix when applied
to the vector $(\ldots,c_{j,k},\ldots,1)^t$ produces a vector of
the degree $d$ coefficients.  For illustration, suppose that $d=3$
and we order our monomials of degree $2$ or less as
$(x,y,x^2,xy,y^2)$ and order the monomials of
degree $3$ as $(x^3,x^2y,xy^2,y^3)$.  Then
\begin{equation}
A =
\begin{bmatrix}
-1 &  0 & -1 &  0 &  0 & 1 \\
-2 & -1 & -1 & -1 &  0 & 3 \\
-1 & -2 &  0 & -1 & -1 & 3 \\
 0 & -1 &  0 &  0 & -1 & 1
\end{bmatrix} .
\end{equation}
Now $A (c_{10},c_{01},c_{20},c_{11},c_{02},1)^t =
(c_{30},c_{21},c_{12},c_{03})^t$.

We need an algorithm to test if a given support
is a support of a sharp polynomial in $\sH(2,d)$.  
We divide the support into the degree $d$ part and the
lower degree part.  We pick all the monomials of degree $d-1$
or less in the support and pick out the corresponding columns in the 
matrix $A$, plus we always take the column corresponding to $(x+y)^d$.
Then we pick out the rows corresponding to the monomials
that do not appear in the support.  We get a submatrix $A'$
and compute its nullspace.  If $A'$ has empty nullspace, no polynomial
with such a support can vanish on the line $x+y=1$.

Any vector in the nullspace of $A'$ represents a set of coefficients of
degree $d-1$ or less in a polynomial that vanishes on $x+y=1$.
We apply these coefficients to $A$ to obtain the degree $d$
coefficients.
The coefficient of
$(x+y)^d$ must not be zero if all the coefficients of the polynomial
are to be positive,
as a polynomial that is zero on $x+y=1$ must have coefficients
of both signs.

We claim that the nullspace must be of dimension exactly 1 if it corresponds
to a sharp polynomial in $\sH(2,d)$.
If the dimension is more than one we would obtain 
a family of polynomials with the given support
by Proposition~\ref{nofamily} the polynomials cannot be sharp.

Therefore we compute the nullspace of $A'$.  If it is of any other
dimension than 1, we are done.  If it is of dimension 1, apply the
corresponding vector to $A$ and test all coefficients for being nonnegative.

As a simplification
we note that except for the last column, $A$ (and hence $A'$) consists of
nonpositive numbers, and the last column is positive.
Hence if
$A'$ contains a row of the form $(0,0,\ldots,0,c)$ for some constant $c$,
the nullspace cannot contain a nonzero vector with only nonnegative entries.

As an example we take the $A$ for degree 3 as given above.  Suppose that
we wish to test the support $xy,x^3,y^3$, that is
\begin{equation}
A' =
\begin{bmatrix}
-1 & 3 \\
-1 & 3 \\
\end{bmatrix} .
\end{equation}
The nullspace of $A'$ is exactly one dimensional and
$(3,1)^t$ spans this space.  Therefore $p(x,y) = (x+y)^3 + 3 b_{11}$ or 
$p(x,y) = 3xy+x^3+y^3$.  Further computation shows that $3xy+x^3+y^3$ is the
only sharp polynomial in $\sH(2,3)$.

When searching for sharp polynomials we apply the following
simplifications.

\begin{enumerate}[(i)]
\item
When the degree $d$ is odd,
the terms of degree $d$ are precisely the terms $x^d+y^d$.
See Lemma~\ref{topdegxdyd}.
\item
Exactly two pure terms occur.  See Proposition~\ref{weakxkym}.
\item
When the degree $d$ is odd,
at least one term of degree $d-1$ appears.
See Lemma~\ref{dminus1}.
\item
When degree $d$ is odd,
terms of the form $x^jy^{d-1-j}$ do not appear for even~$j$.
See Lemma~\ref{dminus1unreachable}.
\item \label{item:noadj}
From any two terms of the form $x^ky^{m+1}$ and $x^{k+1}y^m$,
only one occurs.  See Remark~\ref{noadjrmk} below.
\end{enumerate}

\begin{remark} \label{noadjrmk}
Suppose our polynomial contains
$c_{k,m+1} x^ky^{m+1} + c_{k+1,m} x^{k+1}y^m$, for nonzero $c_{k,m+1}$
and $c_{k+1,m}$. 
Assume without loss of generality that
$c_{k,m+1} < c_{k+1,m}$.  Then
$c_{k,m+1} x^ky^{m+1} + c_{k+1,m} x^{k+1}y^m
= c_{k,m+1} (x+y) x^ky^m + (c_{k+1,m}-c_{k,m+1}) x^{k+1}y^m$.  We can
replace $x+y$ with 1, to obtain a new nonequivalent sharp polynomial.
Once we have found
the second one we would have found the first as well.
No sharp polynomials with this configuration of monomials
have been found so far.
If Proposition~\ref{familyest} below is true for all degrees
then no such sharp polynomials actually exist.  Let $f$ be the polynomial
containing
two monomials as in \eqref{item:noadj}, and $g$ be the polynomial obtained by
dividing out an $(x+y)$.  We would obtain a contradiction by
considering $(1-t)f + tg$.
\end{remark}

The algorithm in this section 
has been implemented using the Genius
software version 1.0.2~\cite{Genius}.  Parts were also implemented in plain C
using the GMP library \cite{GMP}.
To reduce computations using large integer arithmetic, we used
modulo arithmetic first to check if the matrix $A'$ is nonsingular.
A very small prime, $p=19$, was sufficient to eliminate vast
majority of cases.

The tests were run on a recent Intel 3GHz CPU.  For $d=11$
the time used was $0.18$ seconds, for $d=13$ the time was $17$ seconds,
for $d=15$ the time was $33$ minutes,
and finally for $d=17$ the time was $77$ hours.
From these timings it appears that the complexity of this method
grows faster than the mixed linear programming method
described in \S~\ref{section:linprog}.  However, at least with the
current implementations, the method of this section is 
faster for small degrees.

We can also use this method to find even nonsharp polynomials.  In this case
we compute the nullspace, which is now possibly more than one
dimensional.  It is no longer easy to find the subspace generated
by nonnegative vectors.  We get many families of polynomials
that have negative coefficients.
A simple heuristic can eliminate most such families.
The rest are easy to sort through by hand.

In \cite{DL:families},
D'Angelo and the first author proved
the following theorem.  By a $k$-dimensional family of polynomials
we mean a $k$-dimensional polytope in the parameter space.  Recall that
$N(f)$ is the number of distinct monomials of $f$.

\begin{thm}[D'Angelo-Lebl]
Let $\sF \subset \sH(2,d)$ be a $k$-dimensional family,
then for any $f \in \sF$
\begin{equation} \label{unsharpfamilybnd}
d \leq 2\Big(N(f)-k\Big)-3 .
\end{equation}
\end{thm}

The bound \eqref{unsharpfamilybnd} is not sharp.
By using the algorithm to find all polynomials in $\sH(2,d)$ as above, we get
a computer assisted proof of the following improvement in a special case,
and this result is sharp.  That is, no better inequality is possible for $d
\leq 9$ and a 1-dimensional family.

\begin{prop} \label{familyest}
Let $\sF \subset \sH(2,d)$ be a $1$-dimensional family,
$d \leq 9$,
then for any $f \in \sF$
\begin{equation}
d \leq 2N(f)-6 .
\end{equation}
\end{prop}

%%%%%%%%%%%%%%%%%%%%%%%%%%%%%%%%%%%%%%%%%%%%%%%%%%%%%%%%%%%%%%%%%%%%%%%%%%%%%

\section{Mixed linear programming method} \label{section:linprog}

Another approach to computing sharp polynomials in $\sH(2,d)$
involves constraint
satisfaction of mixed-integer programs. That is to say, some variables will
take on integer (actually 0-1) values and others will be continuous. We will
apply classical methods, using branching and a naive form of cutting planes
for handling the integer variables. Good background references are
\cite{Dantzig1963} and \cite{Schrijver1986}.

We now describe in brief the setup of such problems. We are given an odd degree \(d\).  We prescribe that the 
minimal number of nonzero terms is
\(\frac{d+3}{2}\), i.e.\ the polynomial is sharp.
From theory just presented we know several constraints
on such polynomials.  Therefore we also prescribe all the constraints listed
in the last section.

We now assign two variables to each of the monomials not ruled out. One
variable will record whether that monomial goes into the polynomial; it takes on
values 0 or 1. The other is constrained to be nonnegative and corresponds to
the coefficient of the monomial in the sharp polynomial we attempt to construct.
These latter variables satisfy equations that arise from the identity
\(p(x,1-x)-1\equiv 0\); we set the coefficients of terms in every degree to
zero.

As shown earlier, one can multiply terms of the sharp polynomial with
powers of \((x+y)\) to
obtain the homogeneous polynomial in $\sH(2,d)$ of degree \(d\)
(and indeed, we get the
same equations as from zeroing coefficients of \(p(x,1 -x)-1\)).
This procedure in fact
gives upper bounds on each coefficient: if \(x^jy^k\) has coefficient
\(c_{j,k}\) then one can easily show that \(0 \leq c_{j,k} \leq \min \{
\binom{d}{j} , \binom{d}{k} \}\). 
With a bit more work one can deduce a generally stronger inequality: for each
$0 \leq m \leq d-j-k$ we have $c_{j,k} \leq \binom{d}{m}/\binom{d-j-k}{m}$.
Nonnegativity is imposed by the
requirements of our polynomials.
The second inequality is of significant interest,
because it allows us to relate the continuous variables to their discrete
counterparts. Specifically, if we call the corresponding 0-1 variable
\(b_{j,k}\) and call the minimum of the binomial quotients \(m_{j,k}\), then we
have \(c_{j,k}\leq m_{j,k} b_{j,k}\).

To simplify the computations we may impose a few other restrictions;
generally speaking, the more inequality restrictions we impose, the faster
the computation runs.  As before, we can insist that there be no pair of
neighbors of equal degree in the polynomial.  That is, we can impose that
the sharp polynomial does not contain 
$c_{k,m+1} x^ky^{m+1} + c_{k+1,m} x^{k+1}y^m$, for nonzero $c_{k,m+1}$
and $c_{k+1,m}$.  See Remark~\ref{noadjrmk}.

Another restriction is to insist that the sum of coefficients from the left
side of the Newton diagram (that is, monomials with \(\deg (x)>\deg (y)\)) be
larger than or equal to the sum from the right side.  This asymmetry is valid
since we do not care about equivalent polynomials obtained by exchanging
variables.

The actual code begins by setting up linear equations and inequalities based
on the discussion above. We solve the equations, thus eliminating some
variables. We use this result to adjust the inequalities accordingly.
The implementation is a standard branch-and-cut mixed linear programming
code.  We branch on the variables constrained to be integral.

The \textit{Mathematica} code to do what we have described in this section
occupies about 60 lines.
It handles the case $d=9$ in about 1.3 seconds, $d=11$ in 24 seconds, $d=13$
in 9 minutes, $d=15$ in 4.6 hours, and $d=17$ in 186 hours.
These timings are on a fairly recent CPU operating at 3.2 GHz, running
version 6.0.2 of \textit{Mathematica}, with settings to use the COIN-LP library
\cite{CLP} to solve the relaxed linear programs. They are of course also
dependent
on the extent to which the authors have found algorithm simplifications based
on the theory presented for these polynomials. A better understanding of terms
that must or must not arise in such polynomials would almost certainly lead to
algorithmic improvements.

%%%%%%%%%%%%%%%%%%%%%%%%%%%%%%%%%%%%%%%%%%%%%%%%%%%%%%%%%%%%%%%%%%%%%%%%%%%%%

\section{Uniqueness of sharp polynomials} \label{section:sharpmaps}

Partial information about when uniqueness fails can be summarized in the
the following theorem.
This theorem combines the results of this paper with the results of
\cite{DL}.  The
fact that uniqueness holds when $d=9$ was given without proof in~\cite{DL} in
anticipation of the present paper.

\begin{thm} \label{mainDLthm}
Uniqueness holds
when $d=1$, $d=3$, $d=5$, $d=9$, and $d=17$. Uniqueness fails in the following cases:
\begin{enumerate}[(i)]
\item
Suppose $d$ is even. Then uniqueness fails for all $d$.
\item
Suppose $d$ is congruent to $3$ mod $4$. Then uniqueness holds for $d=3$
and fails for $d\ge 7$.
\item
Let $k$ be a positive integer.  Uniqueness fails for $d$ of the form
\begin{equation*}
d = \frac{ (7 + 4\sqrt{3})^{k} + (7 - 4 \sqrt{3})^{k} }{2} ,
\end{equation*}
I.e.\ $d=7, 97, 1351, 18817, 262087, \ldots$
%I.e.\ $d=97, 18817, 3650401, 708158977, 137379191137, \ldots$
\item
Suppose $d>1$ and $d$ is congruent to $1$ mod $6$. Then uniqueness fails.
\end{enumerate}
\end{thm}

The first case not handled by this theorem is $d=21$.
It is computationally infeasible to completely test this case
with the algorithms we have so far.
Nonuniqueness in the theorem is proved by an explicit construction of new sharp
polynomials that is sketched out in the next section.
We can test these constructions more generally using a computer and get further
results on degrees where uniqueness fails.  We get a computer generated
proof of the following proposition, which contains more information than
Theorem~\ref{mainDLthm} for small degrees.

\begin{prop} \label{prop:possibledegs}
Uniqueness fails for all degrees $d \leq 149$ not contained in the
following list:
\begin{equation} \label{possibledegs}
1, 3, 5, 9, 17, 21, 33, 41, 45, 53, 69, 77, 81, 93,
105, 113, 117, 125, 129, 141, 149
%Full list until 513:
%\begin{split}
%& 1, 3, 5, 9, 17, 21, 33, 41, 45, 53, 69, 77, 81, 93,
%105, 113, 117, 125, 129,\\
%& 141, 149,
%153, 161, 165, 177, 185, 201, 213, 221,
%225, 249, 261, 269, 273,\\
%& 285, 297, 305, 309, 333, 341,
%345, 357, 365, 369,
%381, 405, 413, 417, 429,\\
%& 437, 441, 453, 465, 473, 489, 501.
%\end{split}
\end{equation}
More precisely, \eqref{possibledegs} lists all degrees $d \leq 149$
where the procedure of \S~\ref{section:construction} fails to produce
a sharp polynomial besides the group invariant one.
\end{prop}

We have run the computer code for degrees up to $2^9+1$ and the sequence
above does not appear to thin out very rapidly.
It seems reasonable to conjecture
that the sequence is infinite.
Since we also know
that, at least up to degree 17, that the construction of
\S~\ref{section:construction} gives all
sharp examples, it is also reasonable to conjecture that the sequence of
degrees for which uniqueness holds is infinite.
The On-Line Encyclopedia of Integer Sequences~\cite{Sloane} 
did not include the sequence in \eqref{possibledegs}.
It is now entered as A143105.
The partial known sequence of degrees where uniqueness holds was entered
as A143106.  The encyclopedia does not contain any sequence that
is a subsequence of \eqref{possibledegs} and starts with $1,3,5,9,17$.

Note that the beginning of the list (up to 33, with the exceptions of 1 and 21)
are degrees of the form $2^k+1$.  It would be incorrect to assume
that uniqueness holds for such degrees.  In fact, uniqueness fails for
example in degrees $2^6+1$, $2^8+1$ and $2^9+1$.

%%%%%%%%%%%%%%%%%%%%%%%%%%%%%%%%%%%%%%%%%%%%%%%%%%%%%%%%%%%%%%%%%%%%%%%%%%%%%

\section{Construction of new sharp polynomials} \label{section:construction}

Except for the even degree case, all the new noninvariant sharp polynomials
constructed for the proof of
Theorem~\ref{mainDLthm} arise in a similar way.  For the nonuniqueness
in the even degree case 
see~\eqref{allevenex}.  There we construct a sharp polynomial of even degree
$d = d_1 + d_2$ from two sharp polynomials of odd degrees $d_1$ and $d_2$.
We observe that the number of possible sharp polynomials
goes to infinity for even degrees.  We also remark that using this method
we can also construct group invariant sharp polynomials by taking $d_1 = d_2 =
\frac{d}{2}$ and using a group invariant polynomial of degree $\frac{d}{2}$
in~\eqref{allevenex}.

Suppose that $d$ is odd.  Take some even $m < d$ and look at 
$f_m$.  Write 
\begin{equation}
f_m(x,y) = \tilde{f}_m(x,y) - y^m ,
\end{equation}
Now notice that $f_m = 1$ on $x+y = 1$, hence on $x+y = 1$ we have
\begin{equation}
\tilde{f}_m(x,y) = 1 + y^m .
\end{equation}
If we can find a constant $c$ and a monomial $x^jy^k$,
such that $cx^jy^k\tilde{f}_m(x,y)$ has at
least two terms common with $f_d(x,y)$, we can ``replace''
$cx^jy^k\tilde{f}_m(x,y)$ with $cx^jy^k(1+y^m)$.  More explicitly,
we can write
\begin{equation} \label{newexample}
\begin{split}
f(x,y) & := f_d(x,y) + cx^jy^k(1+y^m - \tilde{f}_m(x,y))
\\
& =
f_d(x,y) - cx^jy^k(f_m(x,y) - 1) .
\end{split}
\end{equation}
From the right hand side, it follows that $f(x,y) = 1$ on $x+y=1$.
If we can show that $f$ has all positive coefficients,
then we are finished.

As an example, take $d=7$.  Then $f_7 = x^7 + 7 x^5 y + 14 x^3 y^2 + 7 xy^3 +
y^7$ and $f_2 = x^2 + 2y - y^2$.  Then we get the new sharp polynomial
\begin{equation}
\begin{split}
f(x,y) & := f_7(x,y) - 7x^3y(f_2(x,y) - 1)
\\
& =
x^7 + 7 x^5 y + 14 x^3 y^2 + 7 xy^3 + y^7
-(7x^5y + 14x^3y^2 - 7x^3y^3 -7x^3y)
\\
& =
x^7  + 7x^3y^3 + 7x^3y + 7 xy^3 + y^7 .
\end{split}
\end{equation}

All sharp polynomials known to the authors are obtained by \eqref{newexample}
in one or more steps.  In particular, we know all noninvariant
sharp polynomials of odd degree $d \leq 17$ are constructed this way.
When $d=13$ then two steps are required, that is
we need to repeat the procedure \eqref{newexample} twice,
to obtain one of the noninvariant sharp polynomials.
To find new polynomials this way,
we look for certain ratios of terms in the possible coefficients
of $f_d$, and then verify that the procedure does not introduce
negative terms.
Each known noninvariant sharp polynomial
gives rise to an infinite sequence of degrees
satisfying a certain Pell equation (see \cite{DL}) where the same construction 
applies.  
It is still necessary, however, to check that
all the
coefficients in the new polynomials are nonnegative.  If all the coefficients
are nonnegative, we get an
infinite sequence of new sharp polynomials in different degrees.
For a few specific cases this is essentially what was done in \cite{DL}
to prove the nonuniqueness parts of Theorem~\ref{mainDLthm}.
The computations, while elementary, quickly become long and tedious.
The Pell equation is degenerate in one very specific case
when $d \equiv 1 \pmod{6}$.  Otherwise, the sequence obtained is very sparse
and ``thins out'' very quickly as the degree rises.  See \cite{DL} for more
information. % on these cases and the Pell equation.

We have formulas for the coefficients of terms of $f_d$, that is,
we know that except for $s=0$, the coefficient of $x^{d-2s}y^s$ in $f_d$
is $\frac{d}{s}\binom{d-1-s}{s-1}$.  See for example~\cite{D:book}.
Hence, it is not hard to check via computer for degrees where the construction
produces new examples.  We have run simple computer code to
try this procedure for all possible parameters for $d \leq 513$.
A partial list of degrees where the procedure fails to produce new sharp
polynomials is given in
Proposition~\ref{prop:possibledegs}.

%%%%%%%%%%%%%%%%%%%%%%%%%%%%%%%%%%%%%%%%%%%%%%%%%%%%%%%%%%%%%%%%%%%%%%%%%%%%%

\section{Background in complex geometry} \label{section:CR}

As stated earlier, the primary motivation for this work
comes from CR geometry.  Let us
describe this connection in detail.  For more information about
complex analysis and CR geometry see the book~\cite{D:book}.
There has been much interest recently in the CR geometry community
in studying the complexity
of CR maps between manifolds.  An interesting model case
comes about by studying proper holomorphic maps between unit balls in
$\C^n$.  Due to the symmetries of the unit ball, this
problem has nontrivial connections to many areas of mathematics including
number theory, combinatorics and real algebraic geometry.

To be more precise,
let $\bB_n \subset \C^n$ be the unit ball and
suppose that $F \colon \bB_n \to \bB_N$ is a proper holomorphic map.
When $F$ extends to a continuous map of the closed ball
$\overline{\bB}_n$ then the map is proper if it maps the boundary
of $\bB_n$ to the boundary of $\bB_N$.
It is not hard
to see using elementary complex analysis that if the target dimension $N$
is smaller than the source dimension $n$, then there are no proper maps.
Alexander~\cite{Alexander} proved that
when $n=N$ and $n \geq 2$,
then all proper maps are automorphisms of the unit ball,
i.e.\ linear fractional.
See also the survey \cite{Forstneric:survey}
or the book \cite{D:book} for
more information about the problem.
We will discuss only $n=2$ from now on.

When $F$ extends to the closure of the ball
its restriction to the boundary
defines a CR map
from the unit sphere in $\C^2$ to the unit sphere in $\C^N$.
When $F$ is sufficiently smooth up to the boundary
then Forstneri\v{c}~\cite{Forstneric} proved that $F$ is rational and that 
the degree of $F$ is bounded in terms of $2$ and $N$.
Faran~\cite{Faran} classified all maps when $N=3$.

D'Angelo has made a systematic study of rational proper maps
(see for example \cite{D:korea})
and
classified the polynomial maps~\cite{D:poly} in the
following sense.  If we allow the target dimension to be large enough,
all polynomial
proper maps are obtained by a finite number of operations from
the unique homogeneous map.
D'Angelo conjectured that
\begin{equation} \label{eqcrbnd}
\deg F \leq
2N-3 .
\end{equation}

Further, we will 
assume that $F$ is a monomial map, that is, every component of $F$
is a single monomial.
We now change notation
slightly.  To say that $F$ extends to the boundary is to say
that $\norm{F(z)}^2 = 1$ whenever $\norm{z}^2 = 1$, where $\norm{\cdot}$
is the standard euclidean norm on $\C^N$ and $\C^2$ respectively.  
When $F$ is a monomial map, then we can replace $\abs{z_1}^2$
by the real variable $x$ and $\abs{z_2}^2$ by the real variable $y$.  
$\norm{F(z)}^2$ then becomes a real polynomial in $x$ and $y$ of same
degree as $F$ and with $N$ nonnegative coefficients.  Similarly
$\norm{z}^2$ becomes $x+y$.
Recall that the set of polynomials of degree $d$ with nonnegative
coefficients such that $p(x,y) = 1$ whenever $x+y = 1$ is denoted
by $\sH(2,d)$.
The inequality \eqref{eqmainbound}
proved in~\cite{DKR} therefore proves the conjecture
\eqref{eqcrbnd} in the special case when the map is a monomial
map.  Furthermore, when $d$ is odd,
the polynomials~\eqref{eqinvmaps} induce monomial
proper maps of balls such that $d = 2N-3$.  Hence if the
conjecture~\eqref{eqcrbnd}
is true for all CR maps, then it is sharp in the sense that the bound is
the best possible.
The maps induced by \eqref{eqinvmaps} are group invariant and hence
induce maps of lens spaces.

A complete classification of the monomial maps is an important
first step in classification of all CR maps between spheres,
and is the main motivation for doing the computations in this paper.
For example, Faran's result~\cite{Faran} on the classification of maps from
$\bB_2$ to $\bB_3$ says that the map is equivalent to one of four possible
monomial maps.  In our language of polynomials in $\sH(2,d)$,
these are the polynomials $x+y$, $x+xy+y^2$, $x^2+2xy+x^2$ and $x^3+3xy+y^3$.

The sharp maps arising from $f_d$ are group invariant under the action of a
finite subgroup of the unitary group $U(2)$ generated by
$\left[\begin{smallmatrix}
 \epsilon  & 0 \\
 0 & \epsilon^m
\end{smallmatrix}\right]$ for a primitive root of unity $\epsilon$.
An natural question
to ask is to find all other invariant maps.

Rudin~\cite{Rudin84} proved that a homogeneous proper map of balls
is equivalent up to unitary transformations
to the identity map tensored
with itself $d$ times.  This map is 
invariant under the action of a cyclic group generated by the matrix $\epsilon
I$, where $\epsilon$ is a primitive root of unity and $I$ is the identity.
Forstneri\v{c}~\cite{Forstneric86}
observed that only
fixed-point-free matrix groups could
arise in the context of proper holomorphic maps between balls that are
smooth up to the boundary.
Based on this work
the second author together with
D'Angelo, \cite{Lichtblau92} and \cite{DAngeloLichtblau92}, proved that (up
to linear transformations)
the only groups $\Gamma$ for which there exists a $\Gamma$-invariant rational
proper map of $\bB_2$ to $\bB_N$ are the cyclic fixed-point-free matrix 
groups generated by
$\left[\begin{smallmatrix}
 \epsilon  & 0 \\
 0 & \epsilon^m
\end{smallmatrix}\right]$
where $\epsilon$ is a primitive \(d^{\text{th}}\) root of 1 for
some odd \(d\), \(m =  1,2\), or \(2m = 1\) modulo \(d\).  Therefore,
the maps generated by $f_d$ and the homogeneous maps are the only 
group invariant maps of spheres.

%%%%%%%%%%%%%%%%%%%%%%%%%%%%%%%%%%%%%%%%%%%%%%%%%%%%%%%%%%%%%%%%%%%%%%%%%%%%%

%\bibliographystyle{plain}
\bibliographystyle{alpha}
\bibliography{properll}

\end{document}